\documentclass[preprint,3p,times]{elsarticle}

\usepackage{amssymb}
\usepackage{amsmath}
\usepackage[all,cmtip]{xy}
\usepackage{amsthm}
\usepackage{color}
\usepackage{enumitem}
\usepackage{tikz}
\usepackage{appendix}
\usepackage[colorlinks=true]{hyperref}
\usepackage[hyphenbreaks]{breakurl}
\input diagxy

\journal{}

\theoremstyle{plain}
  \newtheorem{thm}{Theorem}[section]
  \newtheorem{lem}[thm]{Lemma}
  \newtheorem{prop}[thm]{Proposition}
  \newtheorem{cor}[thm]{Corollary}
\theoremstyle{definition}
  \newtheorem{defn}[thm]{Definition}
  \newtheorem{exmp}[thm]{Example}
  \newtheorem{rem}[thm]{Remark}

\newtheorem*{ques}{Question}

\DeclareMathAlphabet{\mathcal}{OMS}{cmsy}{m}{n}

\makeatletter
\def\ps@pprintTitle{%
 \let\@oddhead\@empty
  \let\@evenhead\@empty
  \def\@oddfoot{\vbox{\hsize=\textwidth\footnotesize
  \vskip 8pt
  \copyright 2018. This manuscript version is made available under the CC-BY-NC-ND 4.0 license \url{https://creativecommons.org/licenses/by-nc-nd/4.0/}. The published version is available at \url{https://doi.org/10.1016/j.fss.2018.07.011}.\\
  }}%
  \let\@evenfoot\@oddfoot}
\makeatother

\newcommand{\da}{\downarrow}

\newcommand{\ra}{\rightarrow}

\newcommand{\Ra}{\Rightarrow}

\newcommand{\lda}{\swarrow}
\newcommand{\rda}{\searrow}

\newcommand{\bv}{\bigvee}
\newcommand{\bw}{\bigwedge}

\newcommand{\dv}{\dashv}

\newcommand{\rqa}{\rightsquigarrow}

\newcommand{\opl}{\oplus}

\newcommand{\dbw}{\displaystyle\bw}

\renewcommand{\phi}{\varphi}
\newcommand{\al}{\alpha}
\newcommand{\be}{\beta}

\newcommand{\De}{\Delta}

\newcommand{\ka}{\kappa}

\newcommand{\CQ}{\mathcal{Q}}

\newcommand{\sD}{{\sf D}}

\newcommand{\sQ}{{\sf Q}}

\newcommand{\FG}{\mathfrak{G}}

\newcommand{\FT}{\mathfrak{T}}

\newcommand{\Fb}{\mathfrak{b}}
\newcommand{\Ff}{\mathfrak{f}}

\newcommand{\Cat}{{\bf Cat}}

\newcommand{\Met}{{\bf Met}}

\newcommand{\Sup}{{\bf Sup}}

\newcommand{\ParMet}{{\bf ParMet}}

\newcommand{\ProbMet}{{\bf ProbMet}}
\newcommand{\ProbParMet}{{\bf ProbParMet}}

\newcommand{\QCat}{\CQ\text{-}\Cat}
\newcommand{\sQCat}{\sQ\text{-}\Cat}

\newcommand{\op}{{\rm op}}

\newcommand{\DQ}{\sD\sQ}

\newcommand{\DQCat}{\DQ\text{-}\Cat}

\newcommand{\with}{\mathrel{\&}}
\newcommand{\Gb}{\FG_{\Fb}}
\newcommand{\Gf}{\FG_{\Ff}}

\numberwithin{equation}{section}

\begin{document}

\begin{frontmatter}



\title{Towards probabilistic partial metric spaces: Diagonals between distance distributions}


\author{Jialiang He}
\ead{jialianghe@scu.edu.cn}

\author{Hongliang Lai}
\ead{hllai@scu.edu.cn}

\author{Lili Shen\corref{cor}}
\ead{shenlili@scu.edu.cn}

\cortext[cor]{Corresponding author.}
\address{School of Mathematics, Sichuan University, Chengdu 610064, China}

\begin{abstract}
The quantale of distance distributions is of fundamental importance for understanding probabilistic metric spaces as enriched categories. Motivated by the categorical interpretation of partial metric spaces, we are led to investigate the quantaloid of diagonals between distance distributions, which is expected to establish the categorical foundation of probabilistic partial metric spaces. Observing that the quantale of distance distributions w.r.t. an arbitrary continuous t-norm is non-divisible, we precisely characterize diagonals between distance distributions, and prove that one-step functions are the only distance distributions on which the set of diagonals coincides with the generated down set.
\end{abstract}


\end{frontmatter}




\section{Introduction}

Inspired by Lawvere's pioneering work \cite{Lawvere1973} which presents metric spaces as enriched categories, during the past decades category theory has been playing an important role in the study of metric spaces and their generalizations. Lawvere's construction, in modern terms, is based on the \emph{quantale}
$$[0,\infty]_+=([0,\infty],+,0)$$
whose underlying complete lattice is the extended non-negative real line $[0,\infty]$ equipped with the order ``$\geq$''. Categories enriched in the quantale $[0,\infty]_+$, i.e., sets $X$ equipped with a map $\al:X\times X\to[0,\infty]$ such that
\begin{enumerate}[label=(M\arabic*),leftmargin=2.8em]
\item \label{I-M:d}
    $\al(x,x)=0$, and
\item \label{I-M:t}
    $\al(x,z)\leq\al(y,z)+\al(x,y)$
\end{enumerate}
for all $x,y,z\in X$, are precisely \emph{(generalized) metric spaces}; that is, (classical) metric spaces dropping the requirements of \emph{symmetry} ($\al(x,y)=\al(y,x)$), \emph{finiteness} ($\al(x,y)<\infty$) and \emph{separatedness} ($\al(x,y)=\al(y,x)=0\iff x=y$).

The motivation of this paper comes from two important generalizations of metric spaces, i.e., \emph{probabilistic metric spaces} \cite{Menger1942,Schweizer1983} (also known as \emph{fuzzy metric spaces} \cite{George1994,Kramosil1975}) and \emph{partial metric spaces} \cite{Bukatin2009,Matthews1994}, both of which can be understood as enriched categories.

Probabilistic metric spaces are metric spaces in which the distance is defined as a map $\phi:[0,\infty]\to[0,1]$ with
$$\phi(t)=\bv_{s<t}\phi(s)$$
for all $t\in[0,\infty]$, called a \emph{distance distribution}, rather than a non-negative real number. As discovered by Chai \cite{Chai2009} and investigated by Hofmann-Reis \cite{Hofmann2013a}, (generalized) probabilistic metric spaces are categories enriched in the quantale
$$\De_*=(\De,\otimes_*,\ka_{0,1})$$
of distance distributions w.r.t. a continuous t-norm $*$ on the unit interval $[0,1]$ (see Proposition \ref{De-quantale} for further explanations) which, in elementary words, are precisely sets $X$ equipped with a map $\al:X\times X\times[0,\infty]\to[0,1]$ such that
\begin{enumerate}[label=(ProbM\arabic*),leftmargin=4.6em,start=0]
\item \label{I-ProbM:d}
    $\al(x,y,-):[0,\infty]\to[0,1]$ is a distance distribution,
\item \label{I-ProbM:r}
    $\al(x,x,t)=1$ for all $t>0$, and
\item \label{I-ProbM:t}
    $\al(y,z,r)*\al(x,y,s)\leq\al(x,z,r+s)$
\end{enumerate}
for all $x,y,z\in X$ and $r,s\in[0,\infty]$.

Partial metric spaces are metric spaces in which the distance from a point to itself may not be zero. Explicitly, (generalized) partial metric spaces are sets $X$ equipped a with a map $\al:X\times X\to[0,\infty]$ such that
\begin{enumerate}[label=(PM\arabic*),leftmargin=3.3em]
\item \label{I-PM:d}
    $\al(x,x)\vee \al(y,y)\leq \al(x,y)$, and
\item \label{I-PM:t}
    $\al(x,z)\leq\al(y,z)-\al(y,y)+\al(x,y)$
\end{enumerate}
for all $x,y,z\in X$. As discovered by H{\"o}hle-Kubiak \cite{Hoehle2011a} and Pu-Zhang \cite{Pu2012} and formalized later by Stubbe \cite{Stubbe2014} in a more general setting, although partial metric spaces are not categories enriched in the quantale $[0,\infty]_+$, one may construct a \emph{quantaloid of diagonals} of the quantale $[0,\infty]_+$, usually denoted by $\sD[0,\infty]_+$, such that partial metric spaces are precisely categories enriched in the quantaloid $\sD[0,\infty]_+$.

It is now natural to ask whether it is possible to consider the probabilistic version of partial metric spaces or, equivalently, the partial version of probabilistic metric spaces. As far as we know, this topic has been considered recently by several authors under the name \emph{probabilistic partial metric spaces} or \emph{fuzzy partial metric spaces}; see, e.g. \cite{Amer2016,Sedghi2015,Wu2017,Yue2015,Yue2014}. Following the categorical interpretation of partial metric spaces, this paper aims at a full-scale investigation of the quantaloid
$$\sD\De_*$$
of diagonals of the quantale $\De_*$ of distance distributions w.r.t. a continuous t-norm $*$ on $[0,1]$, so that a categorical foundation of probabilistic partial metric spaces can be established, which was more or less neglected in the existing references.

In general, let $\sQ=(\sQ,\with,1)$ be a commutative and integral quantale (i.e., complete residuated lattice), in which
$$p\with q\leq r\iff p\leq q\ra r$$
for all $p,q,r\in\sQ$. With $\DQ$ denoting the quantaloid of diagonals of $\sQ$ (see Proposition \ref{DQ-def}), a \emph{$\DQ$-category} (also called \emph{partial $\sQ$-category}, see \cite{Hofmann2016}) consists of a set $X$ and a map $\al:X\times X\to\sQ$ such that
\begin{enumerate}[label=(\arabic*)]
\item $\al(x,y)$ is a diagonal between $\al(x,x)$ and $\al(y,y)$, and
\item $(\al(y,y)\ra\al(y,z))\with\al(x,y)=\al(y,z)\with(\al(y,y)\ra\al(x,y))\leq\al(x,z)$
\end{enumerate}
for all $x,y,z\in X$.

By taking $\sQ=[0,\infty]_+$ we will see that $\sD[0,\infty]_+$-categories are exactly partial metric spaces, as pointed out in \cite{Hoehle2011a,Pu2012}. In this case, it is easy to determine diagonals between non-negative real numbers; indeed, $\al(x,y)$ is a diagonal between $\al(x,x)$ and $\al(y,y)$ if, and only if, $\al(x,x)\vee \al(y,y)\leq \al(x,y)$, which is precisely the condition \ref{I-PM:d} for partial metric spaces $(X,\al)$. However, the elegant form of diagonals between non-negative real numbers relies heavily on the fact that the quantale $[0,\infty]_+$ is \emph{divisible}. In fact, in every divisible quantale $\sQ$ it holds that
\begin{equation} \label{I-DQpq-div}
\DQ(p,q)=\,\da\!(p\wedge q)
\end{equation}
for all $p,q\in\sQ$; that is, diagonals between $p$ and $q$ are exactly the down set generated by $p\wedge q$. Unfortunately, Proposition \ref{De-non-div} tells us that the quantale $\De_*$ of distance distributions w.r.t. any continuous t-norm $*$ is non-divisible, and therefore the following question becomes a tough one for the quantale $\De_*$:

\begin{ques}
What are the diagonals between distance distributions?
\end{ques}

With necessary preparations in Section \ref{Preliminaries}, we investigate the quantale $\De_*$ and the quantaloid $\sD\De_*$ thoroughly in Sections \ref{Delta} and \ref{DDelta}, and the main results of this paper answer the above question from two aspects:
\begin{itemize}
\item Theorem \ref{dd-div} characterizes diagonals between any pair of distance distributions in the quantale $\De_*$ w.r.t. an arbitrary continuous t-norm $*$.
\item Theorem \ref{step-only-divisible-lower} shows that the characterization of diagonals obtained from Theorem \ref{dd-div} cannot be simplified to Equation \eqref{I-DQpq-div} unless the involved distance distributions intersect to yield one-step functions, whose proof is the most challenging one in this paper.
\end{itemize}
As an application of these results, in Definition \ref{ProbPM} we propose a rigorous definition of probabilistic partial metric spaces w.r.t. a continuous t-norm $*$ through $\sD\De_*$-categories. Finally, we attach an appendix with some interesting results about the categorical connections between $\sQ$-categories and $\DQ$-categories which, in particular, reveal the interactions between (probabilistic) metric spaces and their partial version.

\section{Diagonals between non-negative real numbers: Partial metric spaces as enriched categories} \label{Preliminaries}


A \emph{(generalized) partial metric space} \cite{Bukatin2009,Matthews1994,Pu2012} is a set $X$ that comes equipped with a map $\al:X\times X\to[0,\infty]$ such that
\begin{enumerate}[label=(PM\arabic*),leftmargin=3.3em]
\item \label{PM:d}
    $\al(x,x)\vee \al(y,y)\leq \al(x,y)$, and
\item \label{PM:t}
    $\al(x,z)\leq\al(y,z)-\al(y,y)+\al(x,y)$
\end{enumerate}
for all $x,y,z\in X$. In particular, a partial metric space $(X,\al)$ satisfying $\al(x,x)=0$ for all $x\in X$ is exactly a \emph{(generalized) metric space} in the sense of Lawvere \cite{Lawvere1973}; that is, a set $X$ equipped with a map $\al:X\times X\to[0,\infty]$ such that
\begin{enumerate}[label=(M\arabic*),leftmargin=2.7em]
\item \label{M:r}
    $\al(x,x)=0$, and
\item \label{M:t}
    $\al(x,z)\leq\al(y,z)+\al(x,y)$
\end{enumerate}
for all $x,y,z\in X$.

\begin{rem} \label{ParMet-classical}
As noted above, our terminology of partial metric spaces here naturally extends Lawvere's notion of generalized metric spaces; so, they are more precisely \emph{generalized} partial metric spaces in the sense of Pu-Zhang \cite{Pu2012}. Besides the conditions \ref{PM:d} and \ref{PM:t}, the notion of ``partial metric'' originally introduced by Matthews \cite{Matthews1994} additionally requires $\al$ to be \emph{symmetric} ($\al(x,y)=\al(y,x)$), \emph{finitary} ($\al(x,y)<\infty$) and \emph{separated} ($\al(x,x)=\al(y,y)=\al(x,y)=\al(y,x)\iff x=y$).
\end{rem}

Although it has been well known for decades that metric spaces can be studied as enriched categories \cite{Lawvere1973} or, more precisely, \emph{quantale}-enriched categories, it is only until recently that partial metric spaces are also understood as enriched categories in a more general framework, where a \emph{quantaloid} is considered as the base for enrichment \cite{Hoehle2011a,Pu2012}.

Explicitly, a \emph{quantaloid} $\CQ$ is a locally ordered category whose hom-sets are complete lattices, such that the composition $\circ$ of $\CQ$-arrows preserves suprema in each variable, i.e.,
$$v\circ\Big(\bv_{i\in I} u_i\Big)=\bv_{i\in I}v\circ u_i\quad\text{and}\quad\Big(\bv_{i\in I} v_i\Big)\circ u=\bv_{i\in I}v_i\circ u$$
for all $\CQ$-arrows $u,u_i:p\to q$, $v,v_i:q\to r$ $(i\in I)$. Hence, the corresponding right adjoints induced by the composition maps
$$-\circ u\dv -\lda u:\ \CQ(p,r)\to\CQ(q,r)\quad\text{and}\quad v\circ -\dv v\rda -:\ \CQ(p,r)\to\CQ(p,q)$$
satisfy
$$v\circ u\leq w\iff v\leq w\lda u\iff u\leq v\rda w$$
for all $\CQ$-arrows $u:p\to q$, $v:q\to r$, $w:p\to r$, where the operations $\lda$ and $\rda$ are called \emph{left} and \emph{right implications} in $\CQ$, respectively.

Given a \emph{small} quantaloid $\CQ$ (i.e., $\CQ$ has a set $\CQ_0$ of objects), a \emph{$\CQ$-category} (or, a \emph{category enriched in $\CQ$}) consists of a set $X$, a map $|\text{-}|:X\to\CQ_0$, and a family of $\CQ$-arrows $\al(x,y)\in\CQ(|x|,|y|)$ $(x,y\in X)$, subject to
$$1_{|x|}\leq \al(x,x)\quad\text{and}\quad \al(y,z)\circ \al(x,y)\leq \al(x,z)$$
for all $x,y,z\in X$. A \emph{$\CQ$-functor} $f:(X,\al)\to(Y,\be)$ between $\CQ$-categories $(X,\al)$, $(Y,\be)$ is a map $f:X\to Y$ such that
$$|x|=|fx|\quad\text{and}\quad \al(x,y)\leq \be(fx,fy)$$
for all $x,y\in X$. The category of $\CQ$-categories and $\CQ$-functors is denoted by $\QCat$.

A one-object quantaloid is precisely a \emph{(unital) quantale}. Let $\with$ denote the multiplication in a quantale $\sQ$, i.e., the composition of arrows in the unique hom-set. We say that
\begin{itemize}
\item $\sQ$ is \emph{commutative} if $p\with q=q\with p$ for all $p,q\in\sQ$;
\item $\sQ$ is \emph{integral} if the unit $1$ of $\sQ$ is also the top element of the complete lattice $\sQ$.
\end{itemize}
Throughout this paper, we let $\sQ=(\sQ,\with,1)$ denote a commutative and integral quantale, which is also known as a \emph{complete residuated lattice}, and we write
$$p\ra q:=q\lda p=p\rda q$$
for implications in $\sQ$. In what follows we are particularly interested in the quantaloid of \emph{diagonals} of $\sQ$:

\begin{prop} \label{DQ-def} (See \cite{Hoehle2011a,Pu2012,Stubbe2014}.)
The following data defines a quantaloid $\DQ$ of  ``diagonals of $\sQ$'':
\begin{itemize}
\item objects of $\DQ$ are elements $p,q,r,\dots$ in $\sQ$;
\item for $p,q\in\sQ$, a morphism $d:p\rqa q$ in $\DQ$, called a \emph{diagonal} from $p$ to $q$, is an element $d\in\sQ$ with
    \begin{equation} \label{diagonal-def}
    (p\ra d)\with p=d=q\with(q\ra d);
    \end{equation}
\item for diagonals $d:p\rqa q$, $e:q\rqa r$, the composition $e\diamond d:p\rqa r$ is given by
    \begin{equation} \label{diagonal-comp}
    e\diamond d=(q\ra e)\with d=e\with(q\ra d);
    \end{equation}
\item $q:q\rqa q$ is the identity diagonal on $q$.
\end{itemize}
\end{prop}

Given $q\in\sQ$, following the terminology in \cite{Hoehle2015}, an element $d\in\sQ$ is said to be \emph{divisible by $q$} if there exists $p\in\sQ$ such that $q\with p=d$. Note that
\begin{equation} \label{diagonal-divisible}
q\with(q\ra d)=d\iff\exists p\in\sQ:\ q\with p=d.
\end{equation}
Hence, from \eqref{diagonal-def} we see that $d:p\rqa q$ is a diagonal if, and only if, $d$ is divisible by $p$ and $q$. In particular, $d$ is divisible by $q$ if, and only if, $d:q\rqa q$ is a diagonal on $q$.

We say that the quantale $\sQ$ is \emph{divisible} \cite{Hajek1998,Hoehle1995a} if $d$ is a diagonal on $q$ (or equivalently, $d$ is divisible by $q$) whenever $d\leq q$ in $\sQ$. It is easy to check that divisible quantales are necessarily integral (see, e.g., the proof of \cite[Proposition 2.1]{Pu2012}).

\begin{rem} \label{diagonal}
The construction of $\DQ$ does not rely on the commutativity or the integrality of $\sQ$. In fact, one may construct the quantaloid $\sD\CQ$ of ``diagonals of $\CQ$'' for any quantaloid $\CQ$ as considered in \cite{Stubbe2014}. Explicitly, a $\CQ$-arrow $d:p_1\to q_2$ with
$$(d\lda u)\circ u=d=v\circ(v\rda d),
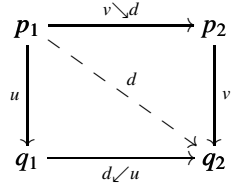
$$
$$\bfig
\square<700,500>[p_1`p_2`q_1`q_2;v\rda d`u`v`d\lda u]
\morphism(0,500)/-->/<700,-500>[p_1`q_2;d]
\efig$$
denoted by $d:u\rqa v$, is called a \emph{diagonal} between $\CQ$-arrows $u$ and $v$, which is precisely the \emph{diagonal} of the above commutative square. The composition of diagonals $d:u\to v$, $e:v\to w$ is given by
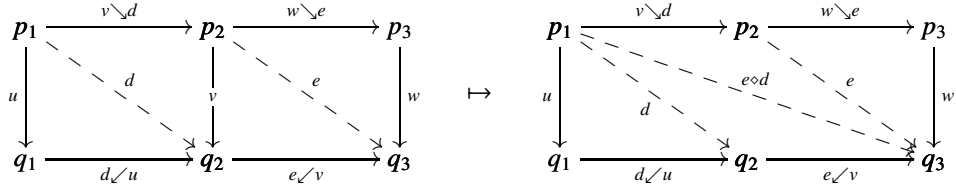
\begin{align*}
e\diamond d&=(e\lda v)\circ d=(e\lda v)\circ(d\lda u)\circ u\\
&=e\circ(v\rda d)=w\circ(w\rda e)\circ(v\rda d).
\end{align*}
$$\bfig
\iiixii<700,500>[p_1`p_2`p_3`q_1`q_2`q_3;v\rda d`w\rda e`u`v`w`d\lda u`e\lda v]
\morphism(0,500)/-->/<700,-500>[p_1`q_2;d]
\morphism(700,500)/-->/<700,-500>[p_2`q_3;e]
\iiixii(2000,0)/->`->`->``->`->`->/<700,500>[p_1`p_2`p_3`q_1`q_2`q_3;v\rda d`w\rda e`u``w`d\lda u`e\lda v]
\morphism(2000,500)/-->/<1400,-500>[p_1`q_3;e\diamond d]
\morphism(2000,500)|b|/-->/<700,-500>[p_1`q_2;d]
\morphism(2700,500)/-->/<700,-500>[p_2`q_3;e]
\place(1700,250)[\mapsto]
\efig$$
It is straightforward to check that $\CQ$-arrows and diagonals constitute a quantaloid $\sD\CQ$. In particular for a commutative and integral quantale $\sQ$, if we denote by $\star$ the single object of the one-object quantaloid $\sQ$, then the commutative square
$$\bfig
\square<700,500>[\star`\star`\star`\star;q\ra d`p`q`p\ra d]
\morphism(0,500)/-->/<700,-500>[\star`\star;d]
\efig$$
illustrates a diagonal $d:p\rqa q$ in $\DQ$ defined by Equation \eqref{diagonal-def}. 
\end{rem}

Applying the definition of $\CQ$-categories to the special case of $\CQ=\DQ$, one sees that a $\DQ$-category consists of a set $X$, a map $|\text{-}|:X\to\sQ$, and a map $\al:X\times X\to\sQ$, such that
\begin{enumerate}[label=(\arabic*)]
\item \label{DQ-Cat:d}
    $\al(x,y)=|x|\with(|x|\ra\al(x,y))=|y|\with(|y|\ra\al(x,y))$, i.e., $\al(x,y)\in\DQ(|x|,|y|)$,
\item \label{DQ-Cat:r}
    $|x|\leq\al(x,x)$,
\item \label{DQ-Cat:t}
    $(|y|\ra\al(y,z))\with\al(x,y)=\al(y,z)\with(|y|\ra\al(x,y))\leq\al(x,z)$
\end{enumerate}
for all $x,y,z\in X$. Since $\sQ$ is integral, it is easy to deduce that
\begin{equation} \label{DQpq}
\DQ(p,q)\subseteq\,\da\!(p\wedge q)
\end{equation}
for all $p,q\in\sQ$, where $\da\!(p\wedge q)$ is the down set generated by $p\wedge q$. Hence, the condition \ref{DQ-Cat:d} together with \eqref{DQpq} implies $\al(x,x)\leq|x|$ for all $x\in X$ and, moreover, the condition \ref{DQ-Cat:r} forces $|x|=\al(x,x)$. Therefore:

\begin{prop} \label{integral-DQCat}
A map $\al:X\times X\to\sQ$ defines a $\DQ$-category structure on a set $X$ if, and only if,
\begin{enumerate}[label={\rm(\arabic*)}]
\item \label{DQ-Cat:s}
    $\al(x,y)=\al(x,x)\with(\al(x,x)\ra\al(x,y))=\al(y,y)\with(\al(y,y)\ra\al(x,y))$, i.e., $\al(x,y)$ is a diagonal between $\al(x,x)$ and $\al(y,y)$,
\item \label{DQ-Cat:tt}
    $(\al(y,y)\ra\al(y,z))\with\al(x,y)=\al(y,z)\with(\al(y,y)\ra\al(x,y))\leq\al(x,z)$
\end{enumerate}
for all $x,y,z\in X$. In this case, one has
$$\al(x,y)\leq\al(x,x)\wedge\al(y,y)$$
for all $x,y\in X$.
\end{prop}

If the quantale $\sQ$ is divisible, then every $d\leq p\wedge q$ is a diagonal between $p,q\in\sQ$, which in conjunction with \eqref{DQpq} forces
\begin{equation} \label{DQpq-div}
\DQ(p,q)=\,\da\!(p\wedge q)
\end{equation}
for all $p,q\in\sQ$; that is, $d:p\rqa q$ is a diagonal if, and only if, $d\leq p\wedge q$. In this case, one may further simplifies Proposition \ref{integral-DQCat} to the following:

\begin{prop} \label{divisible-DQCat} (See \cite{Pu2012}.)
If $\sQ$ is divisible, then a map $\al:X\times X\to\sQ$ defines a $\DQ$-category structure on a set $X$ if, and only if,
\begin{enumerate}[label={\rm(\arabic*)}]
\item \label{DQ-Cat:ss}
    $\al(x,y)\leq\al(x,x)\wedge\al(y,y)$, i.e., $\al(x,y)$ is a diagonal between $\al(x,x)$ and $\al(y,y)$,
\item \label{DQ-Cat:ttt}
    $(\al(y,y)\ra\al(y,z))\with\al(x,y)=\al(y,z)\with(\al(y,y)\ra\al(x,y))\leq\al(x,z)$
\end{enumerate}
for all $x,y,z\in X$.
\end{prop}

In order to exhibit partial metric spaces as $\DQ$-categories, let us now look at the well-known Lawvere's quantale (see \cite{Lawvere1973})
$$[0,\infty]_+=([0,\infty],+,0),$$
where $[0,\infty]$ is the extended non-negative real line equipped with the order ``$\geq$'' (so that $0$ becomes the top element and $\infty$ the bottom element), and ``$+$'' is the usual addition extended via
\begin{equation} \label{infty-add}
p+\infty=\infty+p=\infty
\end{equation}
to $[0,\infty]$. Note that the implication in $[0,\infty]_+$ is given by
$$p\ra q=\begin{cases}
0 & \text{if}\ p\geq q,\\
q-p & \text{if}\ p<q
\end{cases}$$
for all $p,q\in[0,\infty]$, where the subtraction ``$-$'' is extended via
\begin{equation} \label{infty-minus}
\infty-p=\begin{cases}
\infty & \text{if}\ p<\infty,\\
0 & \text{if}\ p=\infty
\end{cases}
\end{equation}
to $[0,\infty]$. Hence, it is easy to see that $[0,\infty]_+$ is a commutative and divisible quantale, and thus a diagonal $d:p\rqa q$ between non-negative real numbers $p,q\in[0,\infty]$ is precisely a real number $d\in[0,\infty]$ with $p\vee q\leq d$. Consequently, the quantaloid
$$\sD[0,\infty]_+$$
of diagonals in $[0,\infty]_+$ has $[0,\infty]$ as the object-set, and its hom-sets are given by
$$\sD[0,\infty]_+(p,q)=\{d\in[0,\infty]\mid p\vee q\leq d\}.$$
By Proposition \ref{divisible-DQCat}, a $\sD[0,\infty]_+$-category structure $\al:X\times X\to[0,\infty]$ on a set $X$ is exactly a partial metric on $X$ described by \ref{PM:d} and \ref{PM:t}. With non-expanding maps $f:(X,\al)\to(Y,\be)$ between partial metric spaces as morphisms, i.e., maps $f:X\to Y$ with
$$\al(x,x)=\be(fx,fx)\quad\text{and}\quad\be(fx,fy)\leq\al(x,y)$$
for all $x,y\in X$, one obtains the category
$$\ParMet:=\sD[0,\infty]_+\text{-}\Cat,$$
which contains the category
$$\Met:=[0,\infty]_+\text{-}\Cat$$
of metric spaces as a full subcategory.

\begin{rem}
In general, $\sQ$-categories are studied as \emph{$\sQ$-preordered sets} in the fuzzy community \cite{Bvelohlavek2004,Denniston2014,Hoehle2015,Lai2006,Lai2009,Shen2013}, while $\DQ$-categories are also referred to as \emph{$\sQ$-preordered $\sQ$-subsets} (or \emph{preordered fuzzy sets}) in the literatures \cite{GutierrezGarcia2017,Pu2012,Shen2016b,Shen2013b}. Therefore, metric spaces are a special kind of $\sQ$-preordered sets, while partial metric spaces can be regarded as examples of $\sQ$-preordered $\sQ$-subsets.
\end{rem}

\section{The quantale of distance distributions w.r.t. a continuous t-norm} \label{Delta}

Probabilistic metric spaces, as defined by \cite{Menger1942,Schweizer1983}, are a generalization of metric spaces in which the distance is defined as a \emph{distance distribution} rather than a non-negative real number, and it has been known in \cite{Chai2009,Clementino2017,Hofmann2013a} that probabilistic metric spaces can be considered as categories enriched in the quantale of distance distributions.

Explicitly, a map $\phi:[0,\infty]\to[0,1]$ is called a \emph{distance distribution} if
$$\phi(t)=\bv_{s<t}\phi(s)$$
for all $t\in[0,\infty]$. In other words, $\phi:[0,\infty]\to[0,1]$ is a distance distribution if, and only if,
\begin{enumerate}[label=(\arabic*)]
\item $\phi(0)=0$,
\item $\phi$ is monotone\footnote{The monotonicity here refers to the usual order ``$\leq$'' on $[0,\infty]$; that is, $s\leq t$ in $[0,\infty]$ implies $\phi(s)\leq\phi(t)$ in $[0,1]$.}, and
\item $\phi$ is left-continuous on $(0,\infty]$.
\end{enumerate}

The set $\De$ of all distance distributions becomes a complete lattice under the pointwise order given by
$$\phi\leq\psi\iff\forall t\in[0,\infty]:\ \phi(t)\leq\psi(t)$$
for all $\phi,\psi\in\De$. The following lemma introduces a canonical procedure of generating distance distributions:

\begin{lem} \label{monotone-dd}
For every map $\phi:[0,\infty]\to[0,1]$, the map
$$\phi^-:[0,\infty]\to[0,1],\quad\phi^-(t)=\bv_{s<t}\bw_{r>s}\phi(r)$$
is a distance distribution. In particular, if $\phi:[0,\infty]\to[0,1]$ is monotone, then
$$\phi^-(t)=\bv_{s<t}\phi(s)$$
for all $t\in[0,\infty]$, and $\phi^-$ is the largest distance distribution that does not exceed the monotone map $\phi$.
\end{lem}

It is often useful to consider a special family of distance distributions, called \emph{one-step functions}, defined by
$$\ka_{p,a}:[0,\infty]\to[0,1],\quad\ka_{p,a}(t)=\begin{cases}
0 & \text{if}\ t\leq p,\\
a & \text{if}\ t>p
\end{cases}$$
for $p\in[0,\infty)$, $a\in[0,1]$.\footnote{We exclude the case of $p=\infty$ in the definition of one-step functions to avoid inconsistency in future discussions; for example, Lemma \ref{De-calc}\ref{De-calc:monotone} may fail to be true if $p=\infty$. Indeed, even if $p=\infty$ is taken into consideration, it is easy to see that $\ka_{\infty,a}=\ka_{0,0}$ for all $a\in[0,1]$.} In particular, $\ka_{0,1}$ and $\ka_{0,0}$ are respectively the top and the bottom elements of $\De$. Moreover, distance distributions can always be written as suprema of one-step functions:

\begin{lem} \label{De-step-join}
For every $\phi\in\De$,
$$\phi=\bv_{a<\phi(p)}\ka_{p,a}=\bv_{p\in[0,\infty)}\ka_{p,\phi(p)}.$$
\end{lem}

Recall that the unit interval $[0,1]$ equipped with a continuous t-norm $*$ \cite{Klement2000,Klement2004} is a commutative and divisible quantale
$$[0,1]_*=([0,1],*,1).$$
It is well known \cite{Faucett1955,Klement2000,Klement2004a,Klement2004b,Mostert1957} that every continuous t-norm $*$ on $[0,1]$ can be written as an ordinal sum of three basic t-norms, i.e., the minimum, the product, and the {\L}ukasiewicz t-norm:
\begin{itemize}
\item (Minimum t-norm) $[0,1]_{\wedge}=([0,1],\wedge,1)$, in which $p\ra_{\wedge} q=\begin{cases}
    1 & \text{if}\ p\leq q\\
    q & \text{if}\ p>q
    \end{cases}$ for all $p,q\in[0,1]$.
\item (Product t-norm) $[0,1]_{\times}=([0,1],\times,1)$, where $\times$ is the usual multiplication on $[0,1]$, and $p\ra_{\times} q=\begin{cases}
    1 & \text{if}\ p\leq q\\
    \dfrac{q}{p} & \text{if}\ p>q
    \end{cases}$ for all $p,q\in[0,1]$. The quantales $[0,1]_{\times}$ and $[0,\infty]_+$ are clearly isomorphic.
\item ({\L}ukasiewicz t-norm) $[0,1]_{\opl}=([0,1],\opl,1)$, in which $p\opl q=\max\{0,p+q-1\}$, and $p\ra_{\opl} q=\min\{1,1-p+q\}$ for all $p,q\in[0,1]$.
\end{itemize}

The \emph{convolution} of monotone maps $\phi,\psi:[0,\infty]\to[0,1]$ w.r.t. a continuous t-norm $*$ is defined as a (necessarily monotone) map $\phi\otimes_*\psi:[0,\infty]\to[0,1]$ with
$$(\phi\otimes_*\psi)(t)=\bv_{r+s=t}\phi(r)*\psi(s).$$

\begin{rem} \label{monotone-convolution-rep}
It is easy to see that
\begin{equation} \label{monotone-convolution-finite}
(\phi\otimes_*\psi)(t)=\bv_{s\leq t}\phi(s)*\psi(t-s)=\bv_{s\leq t}\phi(t-s)*\psi(s)
\end{equation}
for all $t\in[0,\infty)$, but
$$(\phi\otimes_*\psi)(\infty)=\phi(\infty)*\psi(\infty)$$
may not satisfy Equation \eqref{monotone-convolution-finite} since $\infty-\infty=0$; see \eqref{infty-add} and \eqref{infty-minus} for the addition and subtraction involving $\infty$.
\end{rem}

\begin{lem} \label{convolution-dd}
For monotone maps $\phi,\psi:[0,\infty]\to[0,1]$, it holds that
\begin{equation} \label{monotone-convolution-dd}
\phi^-\otimes_*\psi^-=(\phi\otimes_*\psi)^-.
\end{equation}
In particular, if $\phi$ is a distance distribution, then
\begin{equation} \label{monotone-convolution-dd-finite}
(\phi\otimes_*\psi^-)(t)=(\phi\otimes_*\psi)(t)
\end{equation}
for all $t\in[0,\infty)$.
\end{lem}

\begin{proof}
It is obvious that $(\phi^-\otimes_*\psi^-)(0)=(\phi\otimes_*\psi)^-(0)=0$, and
$$(\phi^-\otimes_*\psi^-)(\infty)=\bv_{r<\infty}\phi(r)*\bv_{s<\infty}\psi(s)=\bv_{r+s<\infty}\phi(r)*\psi(s)=\bv_{t<\infty}\bv_{r+s=t}\phi(r)*\psi(s)
=\bv_{t<\infty}(\phi\otimes_*\psi)(t)=(\phi\otimes_*\psi)^-(\infty).$$
It remains to prove the case of $t\in(0,\infty)$. Since $\phi^-\otimes_*\psi^-\leq\phi\otimes_*\psi$, from Lemma \ref{monotone-dd} one immediately obtains $\phi^-\otimes_*\psi^-\leq(\phi\otimes_*\psi)^-$. For the reverse inequality, one needs to show that
$$\bv_{s<t}\bv_{r\leq s}\phi(r)*\psi(s-r)\leq\bv_{p\leq t}\phi^-(p)*\psi^-(t-p)$$
by Equation \eqref{monotone-convolution-finite}. Indeed, whenever $r\leq s<t<\infty$, let $q=r+\dfrac{t-s}{2}$, then $r<q<t$ and $s-r<t-q$. It follows that
$$\phi(r)*\psi(s-r)\leq\Big(\bv_{p<q}\phi(p)\Big)*\Big(\bv_{p<t-q}\psi(p)\Big)=\phi^-(q)*\psi^-(t-q)\leq\bv_{p\leq t}\phi^-(p)*\psi^-(t-p).$$
In particular, if $\phi\in\De$, then $(\phi\otimes_*\psi^-)(0)=(\phi\otimes_*\psi)(0)=0$, and
\begin{align*}
(\phi\otimes_*\psi)(t)&=\bv_{s\leq t}\phi(t-s)*\psi(s)&(\text{Equation \eqref{monotone-convolution-finite}})\\
&=\bv_{s<t}\phi(t-s)*\psi(s)&(\phi\in\De\implies\phi(0)=0)\\
&=\bv_{s<r<t}\phi(r-s)*\psi(s)&(\phi\in\De\implies\phi\ \text{is left-continuous})\\
&=\bv_{r<t}(\phi\otimes_*\psi)(r)\\
&=(\phi\otimes_*\psi)^-(t)\\
&=(\phi\otimes_*\psi^-)(t)&(\text{Equation \eqref{monotone-convolution-dd}})
\end{align*}
for all $t\in(0,\infty)$.
\end{proof}

\begin{rem}
As pointed out by an anonymous referee, Equation \eqref{monotone-convolution-dd-finite} in Lemma \ref{convolution-dd} may not be true for $t=\infty$. For example, let $\phi=\ka_{0,1}$, and let $\psi:[0,\infty]\to[0,1]$ be the monotone map given by
$$\psi(t)=\begin{cases}
0 & \text{if}\ t\in[0,\infty),\\
1 & \text{if}\ t=\infty.
\end{cases}$$
Then $(\phi\otimes_*\psi^-)(\infty)=(\phi\otimes_*\ka_{0,0})(\infty)=\ka_{0,0}(\infty)=0$, but $(\phi\otimes_*\psi)(\infty)=\phi(\infty)*\psi(\infty)=1$.
\end{rem}

If $\phi,\psi\in\De$, it is clear that
\begin{align*}
(\phi\otimes_*\psi)(t)&=\bv_{r+s=t}\phi(r)*\psi(s)=\bv_{s\leq t}\phi(s)*\psi(t-s)=\bv_{s\leq t}\phi(t-s)*\psi(s)\\
&=\bv_{r+s<t}\phi(r)*\psi(s)=\bv_{s<t}\phi(s)*\psi(t-s)=\bv_{s<t}\phi(t-s)*\psi(s)
\end{align*}
for all $t\in[0,\infty]$, and thus $\phi\otimes_*\psi\in\De$. With $\ka_{0,1}$ being the neutral element for $\otimes_*$ one actually defines a quantale structure on $\De$:

\begin{prop} \label{De-quantale}
$\De_*=(\De,\otimes_*,\ka_{0,1})$ is a commutative and integral quantale.
\end{prop}


\begin{rem}
As revealed by \cite{GutierrezGarcia2017a}, the quantale $\De_*$ of distance distributions is the \emph{tensor product} of Lawvere's quantale $[0,\infty]_+$ and the quantale $[0,1]_*$ of continuous t-norm $*$ on $[0,1]$ in the category $\Sup$ of complete lattices and $\sup$-preserving maps.
\end{rem}

A \emph{(generalized) probabilistic metric space} \cite{Chai2009,Hofmann2013a,Hofmann2014} w.r.t. a continuous t-norm $*$ or, equivalently, a $\De_*$-category, is then defined as a set $X$ equipped with a map
$$\al:X\times X\to\De,$$
called the \emph{probabilistic distance function}, subject to the following conditions, for all $x,y,z\in X$ and $r,s\in[0,\infty]$:
\begin{enumerate}[label=(ProbM\arabic*),leftmargin=4.6em]
\item \label{ProbM:r}
    $\al(x,x)(t)=1$ for all $t>0$; that is, $\al(x,x)=\ka_{0,1}$.
\item \label{ProbM:t}
    $\al(y,z)(r)*\al(x,y)(s)\leq\al(x,z)(r+s)$; that is, $\al(y,z)\otimes_*\al(x,y)\leq\al(x,z)$.
\end{enumerate}
With $\De_*$-functors $f:(X,\al)\to(Y,\be)$ as morphisms, i.e., maps $f:X\to Y$ with
$$\al(x,y)(t)\leq\be(fx,fy)(t)$$
for all $x,y\in X$, $t\in[0,\infty]$, we obtain the category
$$\ProbMet_*:=\De_*\text{-}\Cat.$$

\begin{rem} \label{ProbMet-classical}
Similarly as we remarked in \ref{ParMet-classical}, a probabilistic metric space is classically defined as a pair $(X,\al)$ satisfying \ref{ProbM:r}--\ref{ProbM:t} and, additionally, the following conditions \cite{Menger1942,Schweizer1983} for all $x,y\in X$:
\begin{enumerate}[label=(ProbM\arabic*),start=3,leftmargin=4.6em]
\item \label{ProbM:sym}
    (symmetry) $\al(x,y)=\al(y,x)$;
\item \label{ProbM:f}
    (finiteness) $\al(x,y)(\infty)=1$;
\item \label{ProbM:sep}
    (separatedness) $\al(x,y)(t)=\al(y,x)(t)=1$ for all $t>0$ implies $x=y$; that is, $\al(x,y)=\al(y,x)=\ka_{0,1}$ implies $x=y$.
\end{enumerate}
\end{rem}

\section{Diagonals between distance distributions: Probabilistic partial metric spaces as enriched categories} \label{DDelta}

It is now natural to define the partial version of probabilistic metric spaces, i.e., probabilistic partial metric spaces, as categories enriched in the quantaloid $\sD\De_*$ of ``diagonals of $\De_*$''. However, the following fact indicates that the structure of $\sD\De_*$ would be far more complicated than what was described by \eqref{DQpq-div}:

\begin{prop} \label{De-non-div}
The quantale $\De_*=(\De,\otimes_*,\ka_{0,1})$ is non-divisible.
\end{prop}

\begin{proof}
Let $\phi,\xi\in\De$ be given by
$$\phi(t):=\begin{cases}
t & \text{if}\ 0\leq t\leq 1,\\
1 & \text{if}\ t>1
\end{cases}\quad\text{and}\quad
\xi(t):=\begin{cases}
0 & \text{if}\ 0\leq t\leq 1,\\
1 & \text{if}\ t>1.
\end{cases}$$
Then $\xi<\phi$, but there is no $\psi\in\De$ with $\xi=\phi\otimes_*\psi$. Indeed, if such $\psi$ exists, then
$$\xi\Big(\dfrac{4}{3}\Big)=\bv_{s<\frac{4}{3}}\phi\Big(\dfrac{4}{3}-s\Big)*\psi(s)=1.$$
Hence, for each $p\in\Big(0,\dfrac{1}{3}\Big)$ one may find $s_p<\dfrac{4}{3}$ such that
\begin{align*}
\phi\Big(\dfrac{4}{3}-s_p\Big)*\psi(s_p)>1-p&\implies\dfrac{4}{3}-s_p\geq\phi\Big(\dfrac{4}{3}-s_p\Big)>1-p\quad\text{and}\quad\psi(s_p)>1-p\\
&\implies s_p<\dfrac{1}{3}+p<\dfrac{2}{3}\quad\text{and}\quad\psi(s_p)>1-p,
\end{align*}
and consequently
$$\psi\Big(\dfrac{2}{3}\Big)\geq\psi(s_p)>1-p$$
for all $p\in\Big(0,\dfrac{1}{3}\Big)$, which forces $\psi\Big(\dfrac{2}{3}\Big)=1$. It follows that
$$0=\xi(1)=\bv_{s<1}\phi(1-s)*\psi(s)\geq\phi\Big(\dfrac{1}{3}\Big)*\psi\Big(\dfrac{2}{3}\Big)=\dfrac{1}{3}*1=\dfrac{1}{3},$$
giving a contradiction.
\end{proof}

Therefore, although we know from \eqref{DQpq} and the integrality of the quantale $\De_*$ that $\xi:\phi\rqa_*\psi$ is a diagonal of $\De_*$ only if
\begin{equation} \label{xi-leq-phi-psi}
\xi\leq\phi\wedge\psi,
\end{equation}
it requires more efforts to determine which distance distributions below $\phi\wedge\psi$ are actually diagonals between $\phi$ and $\psi$. To achieve this, let us first describe implications in the quantale $\De_*$. As a preparation we list here some rules that are needed in the calculations later on:

\begin{lem} \phantomsection \label{De-calc}
\begin{enumerate}[label={\rm(\arabic*)}]
\item \label{De-calc:step-comp} (See \cite{Hofmann2013a} for the case of $*=\times$)
    $\ka_{p,a}\otimes_*\ka_{q,b}=\ka_{p+q,a*b}$ for all $p,q\in[0,\infty)$, $a,b\in[0,1]$.
\item \label{De-calc:monotone}
    If $\phi:[0,\infty]\to[0,1]$ is a monotone map, then
    $$a\leq\bw_{r>0}\phi(p+r)\iff\ka_{p,a}\leq\phi^-$$
    for all $p\in[0,\infty)$, $a\in[0,1]$.
\item \label{De-calc:imp}
    Let $\phi,\xi,\theta\in\De$. Then $\phi\otimes_*\psi\leq\xi\iff\psi\leq\theta$ for all $\psi\in\De$ if, and only if,
    $$\phi\otimes_*\ka_{p,a}\leq\xi\iff\ka_{p,a}\leq\theta$$
    for all $p\in[0,\infty)$, $a\in[0,1]$.
\end{enumerate}
\end{lem}

\begin{proof}
\ref{De-calc:step-comp} Follows immediately from the definition of one-step functions and convolutions.

\ref{De-calc:monotone} From the definition of one-step functions one soon has $a\leq\dbw_{r>0}\phi(p+r)\iff\ka_{p,a}\leq\phi$, and $\ka_{p,a}\leq\phi\iff\ka_{p,a}\leq\phi^-$ since $\phi^-$ is the largest distance distribution less than or equal to $\phi$ (see Lemma \ref{monotone-dd}).

\ref{De-calc:imp} Since $\phi\otimes_* -$ preserves suprema, one derives the non-trivial direction immediately by Lemma \ref{De-step-join}.
\end{proof}

Let $\ra_*$ denote the implication in the quantale $[0,1]_*$. We start by investigating implications of one-step functions in the quantale $\De_*$:

\begin{prop} \label{step-implication} (See \cite{Hofmann2013a} for the case of $*=\times$)
For any $\xi\in\De$ and $p\in[0,\infty)$, $a\in[0,1]$, the implication $\ka_{p,a}\Ra_*\xi$ in the quantale $\De_*$ is given by
$$(\ka_{p,a}\Ra_*\xi)(t)=\bv_{s<t}a\ra_*\xi(p+s).$$
In particular,
$$\ka_{p,a}\Ra_*\ka_{r,c}=\ka_{\max\{0,r-p\},\,a\ra_* c}$$
for all $r\in[0,\infty)$, $c\in[0,1]$.
\end{prop}

\begin{proof}
The map $\theta:[0,\infty]\to[0,1]$ with $\theta(t)=a\ra_*\xi(p+t)$ is clearly monotone. By Lemma \ref{De-calc}\ref{De-calc:imp} it suffices to prove
$$\ka_{p,a}\otimes_*\ka_{q,b}\leq\xi\iff\ka_{q,b}\leq\theta^-$$
for all $q\in[0,\infty)$, $b\in[0,1]$. Indeed,
\begin{align*}
\ka_{p,a}\otimes_*\ka_{q,b}\leq\xi&\iff\ka_{p+q,a*b}\leq\xi&(\text{Lemma \ref{De-calc}\ref{De-calc:step-comp}})\\
&\iff a*b\leq\bw_{r>0}\xi(p+q+r)&(\text{Lemma \ref{De-calc}\ref{De-calc:monotone}})\\
&\iff b\leq a\ra_*\bw_{r>0}\xi(p+q+r)=\bw_{r>0}\theta(q+r)\\
&\iff\ka_{q,b}\leq\theta^-,&(\text{Lemma \ref{De-calc}\ref{De-calc:monotone}})
\end{align*}
as desired. In particular, if $\xi=\ka_{r,c}$, then for all $q\in[0,\infty)$, $b\in[0,1]$ one has
\begin{align*}
\ka_{p,a}\otimes_*\ka_{q,b}\leq\ka_{r,c}&\iff\ka_{p+q,a*b}\leq\ka_{r,c}&(\text{Lemma \ref{De-calc}\ref{De-calc:step-comp}})\\
&\iff r\leq p+q\ \text{and}\ a*b\leq c\\
&\iff\max\{0,r-p\}\leq q\ \text{and}\ b\leq a\ra_*c\\
&\iff\ka_{q,b}\leq\ka_{\max\{0,r-p\},\,a\ra_* c}.
\end{align*}
Hence $\ka_{p,a}\Ra_*\ka_{r,c}=\ka_{\max\{0,r-p\},\,a\ra_* c}$, again by Lemma \ref{De-calc}\ref{De-calc:imp}.
\end{proof}

Now we turn to the general case. For $\phi,\xi\in\De$, define a map
\begin{equation} \label{rho-def}
\rho_*(\phi,\xi):[0,\infty]\to[0,1],\quad\rho_*(\phi,\xi)(t)=\bw_{q>0}\phi(q)\ra_*\xi(q+t).
\end{equation}
If we interpret $a\ra_* b$ as the ``distance'' of $a,b\in[0,1]$ w.r.t. the t-norm $*$, then $\rho_*(\phi,\xi)(t)$ gives the ``smallest vertical distance'' between the graphs of $\phi$ and the monotone map obtained by shifting $\xi$ to the left for $t$ units. Then the implication in $\De_*$ will be given by
$$\phi\Ra_*\xi=\rho_*(\phi,\xi)^-:$$

\begin{thm} \label{De-implication}
Let $\phi,\xi\in\De$. Then the implication $\phi\Ra_*\xi$ in the quantale $\De_*$ is given by
$$(\phi\Ra_*\xi)(t)=\rho_*(\phi,\xi)^-(t)=\bv_{s<t}\bw_{q>0}\phi(q)\ra_*\xi(q+s)$$
for all $t\in[0,\infty]$.
\end{thm}

\begin{proof}
By Lemma \ref{De-calc}\ref{De-calc:imp} it suffices to prove
$$\phi\otimes_*\ka_{p,a}\leq\xi\iff\ka_{p,a}\leq\rho_*(\phi,\xi)^-$$
for all $p\in[0,\infty)$, $a\in[0,1]$. Indeed,
\begin{align*}
\phi\otimes_*\ka_{p,a}\leq\xi&\iff\forall q\in(0,\infty):\ \ka_{q,\phi(q)}\otimes_*\ka_{p,a}\leq\xi&(\text{Lemma \ref{De-step-join}})\\
&\iff\forall q\in(0,\infty):\ \ka_{p+q,a*\phi(q)}\leq\xi&(\text{Lemma \ref{De-calc}\ref{De-calc:step-comp}})\\
&\iff\forall q\in(0,\infty):\ a*\phi(q)\leq\bw_{r>0}\xi(p+q+r)&(\text{Lemma \ref{De-calc}\ref{De-calc:monotone}})\\
&\iff\forall q\in(0,\infty):\ a\leq\phi(q)\ra_*\bw_{r>0}\xi(p+q+r)\\
&\iff a\leq\bw_{\substack{q>0\\r>0}}\phi(q)\ra_*\xi(p+q+r)=\bw_{r>0}\rho_*(\phi,\xi)(p+r)\\
&\iff\ka_{p,a}\leq\rho_*(\phi,\xi)^-,&(\text{Lemma \ref{De-calc}\ref{De-calc:monotone}})
\end{align*}
which completes the proof.
\end{proof}

Now we are ready to characterize diagonals of the quantale $\De_*$. Recall from \eqref{diagonal-def} that $\xi:\phi\rqa_*\psi$ is a diagonal if, and only if, $\xi:\phi\rqa_*\phi$ is a diagonal on $\phi$ and $\xi:\psi\rqa_*\psi$ is a diagonal on $\psi$; that is, $\xi$ is simultaneously divisible by $\phi$ and $\psi$. Hence, along with \eqref{xi-leq-phi-psi} it suffices to find those distance distributions below a given $\phi\in\De$ that are diagonals on $\phi$. First, let us look at the case when $\phi$ is a one-step function:

\begin{prop} \label{step-divisible}
Let $p\in[0,\infty)$, $a\in[0,1]$. Then $\xi\in\De$ is a diagonal on $\ka_{p,a}$ in the quantale $\De_*$ if, and only if, $\xi\leq\ka_{p,a}$.
\end{prop}

\begin{proof}
It suffices to show that every $\xi\leq\ka_{p,a}$ is divisible by $\ka_{p,a}$. Define $\theta:[0,\infty]\to[0,1]$ with $\theta(t)=a\ra_*\xi(p+t)$ as in the proof of Proposition \ref{step-implication}. Then, by Equation \eqref{monotone-convolution-finite}, for any $t\in[0,\infty)$ one has
$$(\ka_{p,a}\otimes_*\theta)(t)=\bv_{s\leq t}\ka_{p,a}(s)*(a\ra_*\xi(p+t-s))=\bv_{p<s\leq t}a*(a\ra_*\xi(p+t-s))=\bv_{p<s\leq t}\xi(p+t-s)=\xi(t),$$
where the penultimate equality follows from $\xi(p+t-s)\leq a$ and the divisibility of the quantale $[0,1]_*$. Thus
$$(\ka_{p,a}\otimes_*\theta^-)(t)=(\ka_{p,a}\otimes_*\theta)(t)=\xi(t)$$
for all $t\in[0,\infty)$ by Lemma \ref{convolution-dd}. Since $\ka_{p,a}\otimes_*\theta^-$ and $\xi$, as distance distributions, are both left-continuous at $\infty$,
$$(\ka_{p,a}\otimes_*\theta^-)(\infty)=\xi(\infty)$$
follows immediately. Hence $\ka_{p,a}\otimes_*\theta^-=\xi$, showing that $\xi$ is divisible by $\ka_{p,a}$.
\end{proof}

For a general $\phi\in\De$, diagonals on $\phi$ in $\De_*$ usually constitute a proper subset of $\da\phi$ (as we will see in Theorem \ref{step-only-divisible-lower} below), and they are characterized as follows:

\begin{thm} \label{dd-div}
Let $\phi\in\De$. Then $\xi\in\De$ is a diagonal on $\phi$ in the quantale $\De_*$ if, and only if,
$$\xi(t)=\bv_{s<t}\bw_{q>0}\phi(t-s)*(\phi(q)\ra_*\xi(q+s))$$
for all $t\in[0,\infty)$. 
\end{thm}

Note that $\xi\in\De$ is a diagonal on $\phi$ if, and only if, $\xi=\phi\otimes_*(\phi\Ra_*\xi)$, which is equivalent to say that
$$\xi(t)=(\phi\otimes_*(\phi\Ra_*\xi))(t)$$
for all $t\in[0,\infty)$ for the same argument as in the last part of the proof of Proposition \ref{step-divisible}. Therefore, Theorem \ref{dd-div} is an immediate consequence of the following lemma:

\begin{lem} \label{dd-diagonal-comp}
For $\phi,\psi,\xi\in\De$,
$$(\psi\otimes_*(\phi\Ra_*\xi))(t)=\bv_{s<t}\bw_{q>0}\psi(t-s)*(\phi(q)\ra_*\xi(q+s))$$
for all $t\in[0,\infty)$.
\end{lem}

\begin{proof}
From Lemma \ref{convolution-dd} one sees that
$$(\psi\otimes_*(\phi\Ra_*\xi))(t)=(\psi\otimes_*\rho_*(\phi,\xi)^-)(t)=(\psi\otimes_*\rho_*(\phi,\xi))(t)$$
for all $t\in[0,\infty)$, where $\rho_*(\phi,\xi)$ is defined by Equation \eqref{rho-def}. Since
$$(\psi\otimes_*\rho_*(\phi,\xi))(t)=\bv_{s\leq t}\psi(t-s)*\rho_*(\phi,\xi)(s)=\bv_{s<t}\bw_{q>0}\psi(t-s)*(\phi(q)\ra_*\xi(q+s))$$
for all $t\in[0,\infty)$ by Equation \eqref{monotone-convolution-finite}, the conclusion thus follows.
\end{proof}

\begin{exmp}
It follows soon from \eqref{diagonal-divisible} that the convolution
$$\phi\otimes_*\psi$$
is a diagonal between $\phi,\psi\in\De$ in the quantale $\De_*$. In particular,
$$\{\phi\otimes_*\psi\mid\psi\in\De\}$$
gives a set of diagonals on $\phi$. Hence, $\phi\otimes_*\psi$ must satisfy the condition given in Theorem \ref{dd-div}; indeed, for any $t\in[0,\infty)$,
$$\bv_{s<t}\phi(t-s)*\psi(s)=(\phi\otimes_*\psi)(t)\leq\bv_{s<t}\bw_{q>0}\phi(t-s)*(\phi(q)\ra_*(\phi\otimes_*\psi)(q+s))$$
since
$$\forall q,s\in[0,\infty]:\ \phi(q)*\psi(s)\leq(\phi\otimes_*\psi)(q+s)$$
implies
$$\forall s\in[0,\infty]: \psi(s)\leq\bw_{q>0}\phi(q)\ra_*(\phi\otimes_*\psi)(q+s).$$
\end{exmp}

Here we present an alternative characterization for diagonals of the quantale $\De_{\wedge}$. Note that each distance distribution $\phi:[0,\infty]\to[0,1]$ is $\sup$-preserving, and thus admits a right adjoint
$$\phi^{\flat}:[0,1]\to[0,\infty],\quad\phi^{\flat}(a):=\bv\{p\in[0,\infty]\mid\phi(p)\leq a\}.$$

\begin{thm} \label{dd-div-minimum}
Let $\phi\in\De$. Then $\xi\in\De$ is a diagonal on $\phi$ in the quantale $\De_{\wedge}$ if, and only if, there exists $\psi\in\De$ with
$$\xi^{\flat}=\phi^{\flat}+\psi^{\flat}.$$
\end{thm}

\begin{proof}
It suffices to show that $(\phi\otimes_{\wedge}\psi)^{\flat}=\phi^{\flat}+\psi^{\flat}$ for all $\phi,\psi\in\De$. This is true since
\begin{align*}
(\phi\otimes_{\wedge}\psi)(t)\leq a&\iff\forall s\leq t:\ \phi(s)\wedge\psi(t-s)\leq a\\
&\iff\forall s\leq t:\ \phi(s)\leq a\ \ \text{or}\ \ \psi(t-s)\leq a\\
&\iff\forall s\leq t:\ s\leq\phi^{\flat}(a)\ \ \text{or}\ \ t-s\leq\psi^{\flat}(a)\\
&\iff t\leq\phi^{\flat}(a)+\psi^{\flat}(a)
\end{align*}
for all $t\in[0,\infty]$, $a\in[0,1]$, where the last equivalence is valid since
$$\exists s\leq t:\ s>\phi^{\flat}(a)\ \ \text{and}\ \ t-s>\psi^{\flat}(a)\iff t>\phi^{\flat}(a)+\psi^{\flat}(a)$$
trivially holds.
\end{proof}

In fact, as the following theorem shows, Theorem \ref{dd-div} cannot be reduced to Proposition \ref{step-divisible} unless $\phi$ is a one-step function, which also gives a stronger proof for Proposition \ref{De-non-div}:

\begin{thm} \label{step-only-divisible-lower}
Every $\xi\leq\phi$ is a diagonal on $\phi$ in the quantale $\De_*$ if, and only if, $\phi$ is a one-step function.
\end{thm}


To prove this theorem we need the following consequence of the well-known representation theorem of continuous t-norms:

\begin{lem} \label{t-norm-rep} (See \cite{Klement2000,Klement2004b,Mostert1957}.)
Let $*$ be a continuous t-norm on $[0,1]$. Then the set of non-idempotent elements of $*$ in $[0,1]$ is a union of countably many pairwise disjoint open intervals $$\{(b_i,c_i)\mid 0<b_i<c_i<1,\ i\in I,\ I\ \text{is countable}\},$$
and for each $i\in I$, the quantale $([a_i,b_i],*,b_i)$ obtained by restricting $*$ on $[a_i,b_i]$ is either isomorphic to the product t-norm $[0,1]_{\times}$ or isomorphic to the {\L}ukasiewicz t-norm $[0,1]_{\oplus}$.
\end{lem}

\begin{proof}[The proof of Theorem \ref{step-only-divisible-lower}]
Suppose that $\phi$ is not a one-step function, then there exists $p\in(0,\infty)$ with $0<\phi(p)<\phi(\infty)$. We proceed with two cases.

{\bf Case 1.} There exists a strictly increasing sequence $\{a_n\}$ in $[0,1]$ such that each $a_n$ is an idempotent element of $*$ and that $\lim\limits_{n\ra\infty}a_n=\phi(\infty)$. In this case, one may find a positive integer $N$ with $a_N\in(\phi(p),\phi(\infty))$. Note that the set
$$\{t\in(0,\infty)\mid\phi(t)>a_N\}$$
is non-empty by applying the left-continuity of $\phi$ to the point $\infty$, and it has $p$ as a lower bound since $\phi$ is monotone. Thus it makes sense to define
$$q:=\inf\{t\in(0,\infty)\mid\phi(t)>a_N\}.$$
Then $q\in[p,\infty)$, and the left-continuity of $\phi$ guarantees that $\phi(q)\leq a_N$. Hence, $\phi(t)\leq a_N$ for all $t\leq q$ and $a_N<\phi(t)\leq\phi(\infty)$ for all $t>q$. Now define $\xi\in\De$ with
$$\xi(t):=\begin{cases}
0 & \text{if}\ t\leq q,\\
\phi(t) & \text{if}\ t>q.
\end{cases}$$
Then $\xi<\phi$, but there is no $\psi\in\De$ with $\xi=\phi\otimes_*\psi$. Indeed, if such $\psi$ exists, then for every $t>q$,
$$\bv_{s<t}\phi(t-s)*\psi(s)=\xi(t)=\phi(t)>a_N.$$
But $\phi(t-s)\leq a_N$ if $t-s\leq q$, the above inequality then implies
$$a_N<\bv_{s<t-q}\phi(t-s)*\psi(s)\leq\bv_{s<t-q}\psi(s)=\psi(t-q)$$
for all $t>q$; that is, $\psi(t)>a_N$ for all $t>0$. It follows that
$$0=\xi(q)=\bv_{s<q}\phi(q-s)*\psi(s)\geq\bv_{0<s<q}\phi(q-s)*a_N=\phi(q)*a_N\geq\phi(p)*a_N=\phi(p)\wedge a_N=\phi(p)>0,$$
where the penultimate equality follows from the idempotency of $a_N$, giving a contradiction.

{\bf Case 2.} There exists no strictly increasing sequence in $[0,1]$ consisting of idempotent elements of $*$ that approaches to $\phi(\infty)$. In this case, by Lemma \ref{t-norm-rep} one may find idempotent elements $b,c$ of $*$ such that $\phi(\infty)\in(b,c]\subseteq[0,1]$ and that the quantale $([b,c],*,c)$ is isomorphic to $[0,1]_{\times}$ or to $[0,1]_{\oplus}$. Let
$$a:=\dfrac{(b\vee\phi(p))+\phi(\infty)}{2}\in(b\vee\phi(p),\ \phi(\infty)),$$
and let
$$q:=\inf\{t\in(0,\infty)\mid\phi(t)>a\}\in[p,\infty)$$
similarly as in Case 1. Then $\phi(t)\leq a$ for all $t\leq q$ and $a<\phi(t)\leq\phi(\infty)$ for all $t>q$. Now define $\xi\in\De$ with
$$\xi(t):=\begin{cases}
0 & \text{if}\ t\leq q,\\
\phi(t) & \text{if}\ t>q.
\end{cases}$$
Then $\xi<\phi$, but there is no $\psi\in\De$ with $\xi=\phi\otimes_*\psi$. Indeed, if such $\psi$ exists, write
$$a_0:=\bw_{s>q}\phi(s),$$
then for every $t>q$,
$$\bv_{s<t}\phi(t-s)*\psi(s)=\xi(t)=\phi(t)\geq a_0\geq a,$$
where at least one of the last two inequalities is strict. But $\phi(t-s)\leq a$ if $t-s\leq q$, the above inequality then implies
$$a_0\leq\bv_{s<t-q}\phi(t-s)*\psi(s)\leq\bv_{s<t-q}\phi(t)*\psi(s)=\phi(t)*\Big(\bv_{s<t-q}\psi(s)\Big)=\phi(t)*\psi(t-q)$$
for all $t>q$, and consequently
$$a_0\leq\bw_{t>q}\phi(t)*\psi(t-q)=\Big(\bw_{t>q}\phi(t)\Big)*\Big(\bw_{t>q}\psi(t-q)\Big)=a_0*\Big(\bw_{t>q}\psi(t-q)\Big),$$
where the first equality follows from the continuity of $*$. Since $a_0\in[a,c]\subseteq(b,c]$ and the quantale $([b,c],*,c)$ is either isomorphic to $[0,1]_{\times}$ or isomorphic to $[0,1]_{\oplus}$, the above inequality holds only if
$$\bw_{t>0}\psi(t)=\bw_{t>q}\psi(t-q)\geq c;$$
that is, $\psi(t)\geq c$ for all $t>0$. Since $c$ is idempotent, a contradiction arises from
$$0=\xi(q)=\bv_{s<q}\phi(q-s)*\psi(s)\geq\bv_{0<s<q}\phi(q-s)*c =\phi(q)*c\geq\phi(p)*c=\phi(p)\wedge c=\phi(p)>0,$$
which completes the proof.
\end{proof}

For $\phi,\psi\in\De$, since it is easy to extract the condition for $\phi\wedge\psi$ to be a one-step function, the following corollary is an immediate consequence of Theorem \ref{step-only-divisible-lower}:

\begin{cor}
Every $\xi\leq\phi\wedge\psi$ is a diagonal between $\phi$ and $\psi$ in the quantale $\De_*$ if, and only if, there exists $p\in[0,\infty)$, $a\in[0,1]$ such that
\begin{enumerate}[label={\rm(\arabic*)}]
\item either $\phi$ or $\psi$ is constant on $[0,p]$ with value $0$, and
\item either $\phi$ or $\psi$ is constant on $(p,\infty]$ with value $a$.
\end{enumerate}
\end{cor}


With Theorem \ref{dd-div} and Lemma \ref{dd-diagonal-comp} it is now possible to characterize $\sD\De_*$-categories through Proposition \ref{integral-DQCat}, which are precisely probabilistic partial metric spaces:

\begin{defn} \label{ProbPM}
A \emph{(generalized) probabilistic partial metric space} w.r.t. a continuous t-norm $*$ is a set $X$ equipped with a map
$$\al:X\times X\to\De,$$
called the \emph{probabilistic partial distance function}, satisfying the following conditions, for all $x,y,z\in X$:
\begin{enumerate}[label=(ProbPM\arabic*),leftmargin=5.2em]
\item \label{ProbPM:r}
    $\al(x,y)$ is a diagonal between $\al(x,x)$ and $\al(y,y)$ in the quantale $\De_*$; that is,
    \begin{align*}
    \al(x,y)(t)&=\bv_{s<t}\bw_{q>0}\al(x,x)(t-s)*(\al(x,x)(q)\ra_*\al(x,y)(q+s))\\
    &=\bv_{s<t}\bw_{q>0}\al(y,y)(t-s)*(\al(y,y)(q)\ra_*\al(x,y)(q+s))
    \end{align*}
    for all $t\in[0,\infty)$.
\item \label{ProbPM:t}
    $\al(y,z)\otimes_*(\al(y,y)\Ra_*\al(x,y))\leq\al(x,z)$; that is,
    $$\bw_{q>0}\al(y,z)(r)*(\al(y,y)(q)\ra_*\al(x,y)(q+s))\leq\al(x,z)(r+s)$$
    for all $r,s\in[0,\infty)$.
\end{enumerate}
\end{defn}

With $\sD\De_*$-functors $f:(X,\al)\to(Y,\be)$ as morphisms, i.e., maps $f:X\to Y$ with
$$\al(x,x)(t)=\be(fx,fx)(t)\quad\text{and}\quad\al(x,y)(t)\leq\be(fx,fy)(t)$$
for all $x,y\in X$, $t\in[0,\infty]$, one obtains the category
$$\ProbParMet_*:=\sD\De_*\text{-}\Cat,$$
which contains the category $\ProbMet_*=\De_*\text{-}\Cat$ of probabilistic metric spaces as a full subcategory.

\begin{rem}
Similar to Remarks \ref{ParMet-classical} and \ref{ProbMet-classical}, it makes sense to say that a probabilistic partial metric space $(X,\al)$ is
\begin{enumerate}[label=(ProbPM\arabic*),start=3,leftmargin=5.2em]
\item \label{ProbPM:sym}
    \emph{symmetric}, if $\al(x,y)=\al(y,x)$ for all $x,y\in X$;
\item \label{ProbPM:f}
    \emph{finitary}, if $\al(x,y)(\infty)=1$ for all $x,y\in X$;
\item \label{ProbPM:sep}
    \emph{separated}, if $\al(x,x)=\al(y,y)=\al(x,y)=\al(y,x)$ implies $x=y$.
\end{enumerate}
\end{rem}

\appendices

\section{Appendix: $\sQ$-categories vs. $\DQ$-categories}

It is easy to observe the following fact, where the coreflector sends each $\DQ$-category $(X,\al)$ to the set
$$X_1:=\{x\in X\mid\al(x,x)=1\}$$
equipped with the $\sQ$-category structure inherited from $(X,\al)$:

\begin{prop} \label{QCat-coref-DQCat}
$\sQCat$ is a full coreflective subcategory of $\DQCat$. In particular, $\Met$ (resp. $\ProbMet_*$) is a full coreflective subcategory of $\ParMet$ (resp. $\ProbParMet_*$).
\end{prop}

The interaction between $\sQ$-categories and $\DQ$-categories is far more profound than Proposition \ref{QCat-coref-DQCat}. First note that there are \emph{lax homomorphisms} of quantaloids\footnote{A lax homomorphism $f:\CQ\to\CQ'$ of small quantaloids consists of a map $f:\CQ_0\to\CQ'_0$ between the object sets and monotone maps $f:\CQ(p,q)\to\CQ'(fp,fq)$ $(p,q\in\CQ_0)$ between hom-sets, such that
$$1_{fq}\leq f1_q\quad\text{and}\quad fv\circ fu\leq f(v\circ u)$$
for all $q\in\CQ_0$ and composable morphisms $u,v$ of $\CQ$.}
$$\Ff:\DQ\to\sQ,\quad(u:p\rqa q)\mapsto(p\ra u)\quad\text{and}\quad\Fb:\DQ\to\sQ,\quad(u:p\rqa q)\mapsto(q\ra u),$$
which induce two functors
$$\Gf:\DQCat\to\sQCat\quad\text{and}\quad\Gb:\DQCat\to\sQCat,$$
called respectively the \emph{forward globalization} and the \emph{backward globalization} functors \cite{Pu2012,Tao2014}. Explicitly, the forward globalization of a $\DQ$-category $(X,\al)$ is the $\sQ$-category $(X,\Gf\al)$ with
$$\Gf\al(x,y)=\al(x,x)\ra\al(x,y),$$
and the backward globalization of $(X,\al)$ is the $\sQ$-category $(X,\Gb\al)$ with
$$\Gb\al(x,y)=\al(y,y)\ra\al(x,y).$$

While considering $\sQ$ as a $\sQ$-category $(\sQ,\pi)$ with $\pi(p,q)=p\ra q$, one obtains a slice category $\sQCat/\sQ$ whose objects are $\sQ$-functors $f:(X,\al)\to(\sQ,\pi)$, and whose morphisms from $f:(X,\al)\to(\sQ,\pi)$ to $g:(Y,\be)\to(\sQ,\pi)$ are $\sQ$-functors $h:(X,\al)\to(Y,\be)$ making the diagram
\begin{equation} \label{QCat-over-Q}
\bfig
\Vtriangle<500,400>[(X,\al)`(Y,\be)`(\sQ,\pi);h`f`g]
\efig
\end{equation}
commutative. Note that each $\DQ$-category $(X,\al)$ induces a $\sQ$-functor $t_{\al}:(X,\Gf\al)\to(\sQ,\pi)$ with $t_{\al}x=\al(x,x)$; indeed, $t_{\al}$ is a $\sQ$-functor since Proposition \ref{integral-DQCat} implies
$$\Gf\al(x,y)=\al(x,x)\ra\al(x,y)\leq\al(x,x)\ra\al(y,y)=\pi(t_{\al}x,t_{\al}y)$$
for all $x,y\in X$. Moreover:

\begin{prop} \label{Gf-fully-faithful}
Let $(X,\al)$, $(Y,\be)$ be $\DQ$-categories.
\begin{enumerate}[label={\rm(\arabic*)}]
\item \label{Gf-fully-faithful:injective}
    $(X,\al)=(Y,\be)$ if, and only if, $t_{\al}=t_{\be}$.
\item \label{Gf-fully-faithful:ff}
$f:(X,\al)\to(Y,\be)$ is a $\DQ$-functor if, and only if, $f$ is a morphism from $t_{\al}:(X,\Gf\al)\to(\sQ,\pi)$ to $t_{\be}:(Y,\Gf\be)\to(\sQ,\pi)$ in $\sQCat/\sQ$.
\end{enumerate}
\end{prop}

\begin{proof}
\ref{Gf-fully-faithful:injective} If $t_{\al}:(X,\Gf\al)\to(\sQ,\pi)$ and $t_{\be}:(Y,\Gf\be)\to(\sQ,\pi)$ coincide, then $X=Y$, $\al(x,x)=\be(x,x)$ and $\al(x,x)\ra\al(x,y)=\be(x,x)\ra\be(x,y)$ for all $x,y\in X$. It then follows from Proposition \ref{integral-DQCat}\ref{DQ-Cat:s} that
$$\al(x,y)=\al(x,x)\with(\al(x,x)\ra\al(x,y))=\be(x,x)\with(\be(x,x)\ra\be(x,y))=\be(x,y)$$
for all $x,y\in X$. Hence $\al=\be$.

\ref{Gf-fully-faithful:ff} Note that
\begin{align*}
f\ \text{is a}\ \DQ\text{-functor}&\iff\forall x,y\in X:\ \al(x,x)=\be(fx,fx)\ \text{and}\ \al(x,y)\leq\be(fx,fy)\\
&\iff\forall x,y\in X:\ t_{\be}f=t_{\al}\ \text{and}\ \al(x,x)\ra\al(x,y)\leq\be(fx,fx)\ra\be(fx,fy)\\
&\iff\forall x,y\in X:\ t_{\be}f=t_{\al}\ \text{and}\ \Gf\al(x,y)\leq\Gf\be(fx,fy)\\
&\iff f:t_{\al}\to t_{\be}\ \text{is a morphism from}\ \text{in}\ \sQCat/\sQ,
\end{align*}
where we have applied the same method as in \ref{Gf-fully-faithful:injective} to the second equivalence, and thus the conclusion holds.
\end{proof}

From Proposition \ref{Gf-fully-faithful} we obtain a fully faithful functor
$$\Gf^{\dag}:\DQCat\to\sQCat/\sQ,\quad (X,\al)\mapsto(t_{\al}:(X,\Gf\al)\to(\sQ,\pi))$$
that embeds $\DQCat$ in $\sQCat/\sQ$ as a full subcategory. Furthermore, this embedding is reflective if $\sQ$ is divisible:

\begin{thm} \label{DQCat-ref-QCatQ}
If $\sQ$ is divisible, then $\DQCat$ is a full reflective subcategory of $\sQCat/\sQ$.
\end{thm}

\begin{proof}
{\bf Step 1.} For each $\sQ$-category $(X,\al)$ equipped with a $\sQ$-functor $f:(X,\al)\to(\sQ,\pi)$,
$$\al_f(x,y)=\al(x,y)\with fx$$
defines a $\DQ$-category $(X,\al_f)$. To see this, it suffices to verify that $(X,\al_f)$ satisfies the conditions given in Proposition \ref{divisible-DQCat}.

First, since $f$ is a $\sQ$-functor, $\al(x,y)\leq\pi(fx,fy)=fx\ra fy$, and thus $\al_f(x,y)=\al(x,y)\with fx\leq fy$ for all $x,y\in X$. It follows that $\al_f(x,y)\leq fx\wedge fy=\al_f(x,x)\wedge\al_f(y,y)$ for all $x,y\in X$. Here
\begin{equation} \label{alf=f}
\al_f(x,x)=fx
\end{equation}
for all $x\in X$ since, with $(X,\al)$ being a $\sQ$-category, one always has $\al(x,x)=1$.

Second, since $(X,\al)$ is a $\sQ$-category, one also has $\al(y,z)\with\al(x,y)\leq\al(x,z)$ for all $x,y,z\in X$, and hence
\begin{align*}
\al_f(y,z)\with(\al_f(y,y)\ra\al_f(x,y))&=\al(y,z)\with fy\with(fy\ra(\al(x,y)\with fx))\\
&\leq\al(y,z)\with\al(x,y)\with fx\leq\al(x,z)\with fx=\al_f(x,z).
\end{align*}

{\bf Step 2.} For each commutative diagram \eqref{QCat-over-Q} in $\sQCat$, $h:(X,\al_f)\to(Y,\be_g)$ is a $\DQ$-functor, hence the above assignment defines a functor $\FT:\sQCat/\sQ\to\DQCat$. Indeed, for all $x\in X$, $y\in Y$,
$$\be_g(hx,hx)=ghx=fx=\al_f(x,x)$$
follows from \eqref{alf=f}, and
$$\al_f(x,y)=\al(x,y)\with fx\leq\be(hx,hy)\with ghx=\be_g(hx,hy)$$
since $h:(X,\al)\to(Y,\be)$ is a $\sQ$-functor.

{\bf Step 3.} $\FT:\sQCat/\sQ\to\DQCat$ is a left adjoint of $\Gf^{\dag}:\DQCat\to\sQCat/\sQ$. To this end, we establish a bijection
\begin{equation} \label{DQCat-cong-sQCatQ}
\DQCat((X,\al_f),(Y,\be))\cong\sQCat/\sQ(f,t_{\be})
\end{equation}
for each $\sQ$-functor $f:(X,\al)\to(\sQ,\pi)$ and $\DQ$-category $(Y,\be)$. Indeed,
\begin{align*}
&h:(X,\al_f)\to(Y,\be)\ \text{is a}\ \DQ\text{-functor}\\
\iff{}&\forall x,y\in X:\ \al_f(x,x)=\be(hx,hx)\ \text{and}\ \al_f(x,y)\leq\be(hx,hy)\\
\iff{}&\forall x,y\in X:\ fx=t_{\be}hx\ \text{and}\ \al(x,y)\with fx\leq\be(hx,hy)&(\text{definition of}\ t_{\be},\ \al_f\ \text{and}\ \eqref{alf=f})\\
\iff{}&\forall x,y\in X:\ fx=t_{\be}hx\ \text{and}\ \al(x,y)\leq\be(hx,hx)\ra\be(hx,hy)&(\text{definition of}\ t_{\be})\\
\iff{}&\forall x,y\in X:\ fx=t_{\be}hx\ \text{and}\ \al(x,y)\leq\Gf\be(hx,hy)\\
\iff{}&h:f\to t_{\be}\ \text{is a morphism in}\ \sQCat/\sQ.
\end{align*}
Finally, it is straightforward to check that the bijection \eqref{DQCat-cong-sQCatQ} is natural in $f:(X,\al)\to(\sQ,\pi)$ and $(Y,\be)$, and thus the proof is completed.
\end{proof}

Let $\sQ^{\op}=(\sQ,\pi^{\op})$ be the $\sQ$-category with the underlying set $\sQ$ and the $\sQ$-category structure $\pi^{\op}(p,q)=q\ra p$. Then similarly one obtains another full embedding
$$\Gb^{\dag}:\DQCat\to\sQCat/\sQ^{\op}$$
that sends each $\DQ$-category $(X,\al)$ to the $\sQ$-category $(X,\Gb\al)$ equipped with the $\sQ$-functor $s_{\al}:(X,\Gb\al)\to(\sQ,\pi^{\op})$ with $s_{\al}x=\al(x,x)$ for all $x\in X$, and this embedding is also reflective when $\sQ$ is divisible:

\begin{thm} \label{DQCat-ref-QCatQop}
If $\sQ$ is divisible, then $\DQCat$ is a full reflective subcategory of $\sQCat/\sQ^{\op}$.
\end{thm}

In particular, if $\sQ=[0,\infty]_+$, then $[0,\infty]=([0,\infty],\pi)$ and $[0,\infty]^{\op}=([0,\infty],\pi^{\op})$ are both metric spaces with
$$\pi(p,q)=\begin{cases}
0 & \text{if}\ p\geq q,\\
q-p & \text{if}\ p<q
\end{cases}\quad\text{and}\quad\pi^{\op}(p,q)=\begin{cases}
0 & \text{if}\ q\geq p,\\
p-q & \text{if}\ q<p
\end{cases}$$
for all $p,q\in[0,\infty]$. In this case, it is easy to see that the functors $\Gf^{\dag}$ and $\FT$ are inverses to each other, and thus
$$\Gf^{\dag}:\ParMet\to\Met/[0,\infty]$$
gives an isomorphism of categories; and so is the functor $\Gb^{\dag}:\ParMet\to\Met/[0,\infty]^{\op}$:

\begin{thm} \label{ParMet-Met-iso}
$\ParMet\cong\Met/[0,\infty]\cong\Met/[0,\infty]^{\op}$.
\end{thm}

If $\sQ=\De_*$, then $\De_*=(\De,\pi_*)$ and $\De_*^{\op}=(\De,\pi_*^{\op})$ are both probabilistic metric spaces with
$$\pi_*(\phi,\psi)(t)=\bv_{s<t}\bw_{q>0}\phi(q)\ra_*\psi(q+s)\quad\text{and}\quad\pi_*^{\op}(\phi,\psi)(t)=\bv_{s<t}\bw_{q>0}\psi(q)\ra_*\phi(q+s)$$
for all $\phi,\psi\in\De$, $t\in[0,\infty]$, and in this case:

\begin{cor}
$\ProbParMet_*$ is a full subcategory of both $\ProbMet_*/\De_*$ and $\ProbMet_*/\De_*^{\op}$.
\end{cor}

\section*{Acknowledgement}

The first author acknowledges the support of National Natural Science Foundation of China (No. 11771311). The second author acknowledges the support of National Natural Science Foundation of China (No. 11771310). The third author acknowledges the support of National Natural Science Foundation of China (No. 11701396) and the Fundamental Research Funds for the Central Universities (No. YJ201644).

The authors are grateful to the anonymous referees for their helpful remarks and suggestions.





\end{document}